\documentclass{amsart}
\usepackage{amsfonts}
\usepackage{amssymb}
\usepackage{times}
\def\ud{\;\dot{\cup} \;}
{\theoremstyle{plain}%
  \newtheorem{theorem}{Theorem}[section]
  
  \newtheorem{proposition}{Proposition}[section]
  \newtheorem{lemma}{Lemma}[section]

  \newtheorem{definition}{Definition}[section]

\newtheorem{remark}{Remark}[section]

\newfont{\hueca}{msbm10}

\begin{document}

\title{Leibniz triple systems admitting a multiplicative basis}

\thanks{The first, second and fourth authors acknowledge financial assistance by the Centre for Mathematics of the University of Coimbra -- UID/MAT/00324/2013, funded by the Portuguese Government through FCT/MEC and co-funded by the European Regional Development Fund through the Partnership Agreement PT2020. Third and fourth authors are supported by the PCI of the UCA `Teor\'\i a de Lie y Teor\'\i a de Espacios de Banach', by the PAI with project numbers FQM298, FQM7156 and by the project of the Spanish Ministerio de Educaci\'on y Ciencia  MTM2013-41208P. The fourth author acknowledges the Funda\c{c}\~{a}o para a Ci\^{e}ncia e a Tecnologia for the grant with reference SFRH/BPD/101675/2014.}

\author[H. Albuquerque ]{Helena~Albuquerque}

\address{Helena~Albuquerque, CMUC, Departamento de Matem\'atica, Universidade de Coimbra, Apartado 3008,
3001-454 Coimbra, Portugal. \hspace{0.1cm} {\em E-mail address}: {\tt lena@mat.uc.pt}}

\author[E. Barreiro]{Elisabete~Barreiro}

\address{Elisabete~Barreiro, CMUC, Departamento de Matem\'atica, Universidade de Coimbra, Apartado 3008,
3001-454 Coimbra, Portugal. \hspace{0.1cm} {\em E-mail address}: {\tt mefb@mat.uc.pt}}{}

\author[A.J. Calder\'on]{A.J.~Calder\'on}

\address{A.J.~Calder\'on, Departamento de Matem\'aticas, Universidad de C\'adiz, Campus de Puerto Real, 11510, Puerto Real, C\'adiz, Espa\~na. \hspace{0.1cm} {\em E-mail address}: {\tt ajesus.calderon@uca.es}}{}

\author[Jos\'{e} M. S\'{a}nchez-Delgado]{Jos\'{e} M. S\'{a}nchez-Delgado}

\address{Jos\'{e} M. S\'{a}nchez-Delgado, CMUC, Departamento de Matem\'atica, Universidade de Coimbra, Apartado 3008, 3001-454 Coimbra, Portugal. \hspace{0.1cm} {\em E-mail address}: {\tt txema.sanchez@mat.uc.pt}}

\begin{abstract} Let $(T,\langle \cdot, \cdot, \cdot \rangle)$ be a Leibniz triple system of arbitrary dimension, over an arbitrary base field ${\mathbb F}$. A basis ${\mathcal B} = \{e_{i}\}_{i \in I}$ of $T$ is called multiplicative if for any $i,j,k \in I$ we have that  $\langle e_i,e_j,e_k\rangle\in {\mathbb F}e_r$ for some $r \in I$. We show that if $T$ admits a multiplicative basis then it decomposes as the orthogonal  direct sum $T= \bigoplus_k{\mathfrak I}_k$ of well-described ideals ${\mathfrak I}_k$ admitting each one a multiplicative basis. Also the minimality of $T$ is characterized in terms of the multiplicative basis and it is shown that, under a mild condition, the above direct sum is by means of the family of its minimal ideals.

\medskip



\end{abstract}

\maketitle

\section{Introduction and previous definitions}

The notion of Leibniz algebra was introduced by Loday \cite{Loday} as a ``non-antisymmetric" generalization of Lie algebras. So far, many results on this kind of algebras have been considered in the framework of low dimensional algebras, nilpotence and related problems (see for instance \cite{Leibniz2,Leibniz4,le18,le16,le24,huel3}).

Leibniz triple systems were recently introduced by Bremner and S\'anchez-Ortega in the paper \cite{Bremner}, applying the Kolesnikov-Pozhidaev algorithm to Lie triple systems.
Furthermore, Leibniz triple systems are related to Leibniz algebras in the same way that Lie triple systems are related to Lie algebras. So it is natural to prove the analogue of results from the theory of Lie triple systems to Leibniz triple systems.

Recently, for instance in \cite{re1,m5,Ref,Yo_Arb_Alg}, the algebraic structures admitting a multiplicative basis have been studied
 and our work can be considered as a generalization of \cite{Arb_Leibniz_Super}.

%\begin{definition}\rm
%A {\it right Leibniz algebra} $L$ is a vector space over a base field $\mathbb{F}$ endowed with a bilinear product $[\cdot,\cdot]$ satisfying the {\it Leibniz identity} $$[[y,z], x] = [[y, x], z] + [y, [z, x]],$$ for all $x, y, z \in L$.
%\end{definition}

\begin{definition}\rm \label{deffi}
A {\it Leibniz triple system} is a vector space $T$ endowed with a trilinear ope\-ration $\{\cdot, \cdot, \cdot\} : T \times T \times T \longrightarrow T$ satisfying

\begin{enumerate}
\item[1.] $\{a, \{b, c, d\}, f\} + \{a, \{c, b, d\}, f\} = 0,$

\item[2.]  $\{a, \{b, c, d\}, f\} + \{a, \{c, d, b\}, f\} + \{a, \{d, b, c\}, f\} = 0,$

\item[3.]
$\{a,b,\{c,d,f\}\} - \{\{a,b,c\},d,f\} + \{\{a,b,d\},c,f\} -
 \{\{a,b,f\},d,c\} +$\\
$ \{\{a,b,f\},c,d\} = 0,$

\item[4.]  $\{\{c,d,f\},b,a\} - \{\{c,d,f\},a,b\} - \{\{c,b,a\},d,f\} +
 \{\{c,a,b\},d,f\} -$ \\
 $ \{c,\{a,b,d\},f\} - \{c,d,\{a,b,f\}\} = 0,$
\end{enumerate}
or equivalently (see Lemma 8 in \cite{Bremner}) satisfying

{\small
\begin{enumerate}
\item[1.] $
\{a,\{b,c,d\},f\} = \{\{a,b,c\},d,f\} - \{\{a,c,b\},d,f\} - \{\{a,d,b\},c,f\} + \{\{a,d,c\},b,f\},$

\item[2.] $\{a,b,\{c,d,f\}\} = \{\{a,b,c\},d,f\} - \{\{a,b,d\},c,f\} - \{\{a,b,f\},c,d\} + \{\{a,b,f\},d,c\},$
\end{enumerate}}
for all $a,b,c,d,f \in T$.
\end{definition}

\noindent The most trivial examples of  Leibniz triple systems are  Lie triple systems with the same ternary product. If $L$ is a Leibniz algebra with product $[\cdot, \cdot]$ then $L$ becomes a Leibniz triple system considering $\{x,y,z\} := [[x, y], z]$, for $x,y,z \in L$. More examples refer to \cite{Bremner}. Throughout this paper, Leibniz triple systems $T$ are considered of arbitrary dimension and over an arbitrary base field $\mathbb{F}$.

A {\it subtriple} $S$ of a Leibniz triple system $T$ is a linear subspace such that $\{S,S,S\} \subset S$. A linear subspace $\mathcal{I}$ of $T$ is called an {\it ideal} if $\{\mathcal{I},T,T\} + \{T,\mathcal{I},T\} + \{T,T,\mathcal{I}\} \subset \mathcal{I}.$ We also recall that two nonzero ideals $\mathcal{I}, \mathcal{J}$ of $T$ are called {\it orthogonal} if the condition $$\{T,\mathcal{I},\mathcal{J}\} + \{\mathcal{I},T,\mathcal{J}\} +\{\mathcal{I},\mathcal{J},T\} + \{T,\mathcal{J},\mathcal{I} \} + \{\mathcal{J},T,\mathcal{I}\} + \{\mathcal{J},\mathcal{I},T \}=0$$ holds. A direct sum $\bigoplus_{j \in J}\mathcal{I}_j$ of non-zero  ideals of $T$ is called an {\it orthogonal direct sum} if $\mathcal{I}_j$ and $\mathcal{I}_k$ are orthogonal ideals for any $j, k \in J$ with $j \neq k$.

The ideal ${\frak I}$ generated by
\begin{equation}\label{gentle}
\{\{x,y,z\} - \{x,z,y\} + \{y,z,x\} : x, y, z \in T\}
 \end{equation}
 plays an important role in the theory since it determines the (possible) non-Lie character of $T$. From definition of ideal $\{{\frak I}, T, T\} \subset {\frak I}$ and from definition of Leibniz triple system it is straightforward to check that this ideal satisfies
\begin{equation}\label{equi}
\{T, {\frak I}, T\} = \{T, T, {\frak I}\} = 0.
\end{equation}

By the above observation, an adequate definition of simplicity in the case of Leibniz triple systems is the corresponding to the ones satisfying that its only ideals are $\{0\}, T$ and the one generated by the set $\{\{x,y,z\} - \{x,z,y\} + \{y,z,x\} : x, y, z \in T\}$. Following this vain, we consider the next definition.

\begin{definition}\label{Defsimple}\rm
A Leibniz triple system $T$ is said to be {\it simple} if $\{T,T,T\} \neq 0$ and its only ideals are $\{0\}, \frak I$ and $T$.
\end{definition}

\noindent Observe that this definition agrees with the definition of a simple Lie triple system, since in this case we have $\frak I = \{0\}$.

Therefore we can write $$T = {\frak I} \oplus {\neg \frak I},$$ where ${\neg \frak I}$ is the linear complement of ${\frak I}$ in $T$.
%Therefore any ideal $\mathcal{I}$ of $T$ have necessarily a non empty intersection with $\frak I$.
Hence, by taking ${\mathcal B}_{\frak I} = \{e_n\}_{n \in I}$ and ${\mathcal B}_{\neg \frak I} = \{u_r\}_{r \in J}$ bases of ${\frak I}$ and ${\neg \frak I}$ respectively, we get $${\mathcal B} = {\mathcal B}_{\frak I} \ud {\mathcal B}_{\neg \frak I}$$ is a basis of $T.$

\begin{definition}\label{11}\rm
A basis ${\mathcal B} = \{v_k\}_{k \in K}$ of $T$ is said to be {\it multiplicative} if for any $k_1,k_2,k_3 \in K$ we have either $\{v_{k_1},v_{k_2},v_{k_3}\} = 0$ or $0 \neq \{v_{k_1},v_{k_2},v_{k_3}\} \in \mathbb{F}v_k$ for some (unique) $k \in K$.
\end{definition}

\begin{remark}\rm
The definition of multiplicative basis given in Definition
\ref{11} is  a little bit more general than the usual one considered in the literature for the case of algebras (\cite{re1,m5,Ref}). In fact, in these references a basis ${\mathcal B}=\{v_k\}_{k \in K}$ of an algebra $A$ is called {\it multiplicative} if for any $k_1,k_2 \in K$ we have either $v_{k_1}v_{k_2}=0$ or $0 \neq v_{k_1}v_{k_2} = v_k$ for some $k \in K$.
\end{remark}

As examples of Leibniz triple systems admitting a multiplicative basis we have the Lie triple systems with multiplicative basis introduced in \cite{Yo_simple_triple, Yo_integrable, Yo_split_triple, Yo_Forero2}.  Jacobson presented in \cite[page 312]{Jacobson} the three isomorphisms classes of 2-dimensional Lie triple systems over an algebraically closed field of characteristic not two. By omitting the one of null products, we have two cases, where zero products are omitted and the basis is ${\mathcal B} = \{x,y\}$. These are: $$\{x,y,x\}=y, \hspace{0.3cm} \{y,x,x\}=-y$$ and $$\{x,y,x\}=2x, \hspace{0.3cm} \{y,x,x\}=-2x, \hspace{0.3cm} \{x,y,y\}=-2y, \hspace{0.3cm} \{y,x,y\}=2y.$$
Clearly, these are examples of Leibniz triple systems admitting  multiplicative bases. Moreover, all of the two-dimensional Leibniz triple systems presented in \cite[\S 8]{Bremner} also admit multiplicative bases.

Since any Leibniz algebra $(L,[\cdot, \cdot])$ with a multiplicative basis ${\mathcal B}$  gives rise to  a Leibniz triple system of triple product  $\{x,y,z\} := [[x, y], z]$ admitting also ${\mathcal B}$ as multiplicative basis, and many of the classes of Leibniz algebras which have been described in the literature, have been made by presenting a multiplication table of the Leibniz algebra in terms of a multiplicative basis, we get a lot of examples of Leibniz triple systems
admitting a multiplicative basis. This is the case for instance of the Leibniz triple systems associated to the two and three dimensional nilpotent Leibniz algebras (see \cite{Leibniz2,Loday}), to the  non-Lie Leibniz algebras $L$ with $L/{\mathcal S}$ abelian described in \cite{Leibniz2},  to the categories of finite-dimensional 0-filiform Leibniz algebras, of finite-dimensional nonsplit graded filiform Leibniz algebras, and of different types of finite-dimensional 2-filiform nonsplit Leibniz algebras (see \cite{huel2}). Also this is the case of the class of four-dimensional solvable Leibniz algebras with three-dimensional rigid nilradical (see \cite{le16}), to the families  of four-dimensional solvable Leibniz algebras with two-dimensional nilradical and of certain types, respect to its radical, of four-dimensional solvable Leibniz algebras  (see \cite{le18}), to several types of solvable Leibniz algebras with naturally graded filiform nilradical considered in \cite{le22}, to the solvable Leibniz algebras whose nilradical is $NF_n$ (see \cite{le24}), to the family of (complex)
finite-dimensional Leibniz algebras with Lie quotient  $sl_2$ (see \cite{huel3}), and so on.

%Let us finally show the following
%example, where we denote by ${\mathbb N}$ the set of non-negative
%integers.

%\begin{example}\label{r3}\rm
% Let
%$T$ be the Leibniz triple system, over a base
%field with characteristic different to $2$,  with basis
%$${\mathcal B}= \{e_n: n \in {\mathbb N} \} \dot{\cup}\{v_a,v_b,v_c,
%v_d\}$$  and where  the non-zero products respect to the
%basis elements  are:

%$$\hbox{$\{v_d,v_d, v_d\}=e_1;$  $ \{e_n,v_a, v_a\}=e_{n}$ for
%$n\geq 2$;}$$

%$$\hbox{$\{e_n,v_a, v_b\}=e_{n+1}$ for
%$n\geq 2$;  $ \{e_n,v_c, v_a\}=(n-2)e_{n-1}$ for
%$n\geq 3$;}$$

%$$\hbox{$\{e_n,v_c, v_b\}=(n-2)e_{n}$ for
%$n\geq 3$;  $ \{e_n,v_c, v_c\}=(n-2)(n-3)e_{n-2}$ for
%$n\geq 4$}.$$

%Clearly, $T$ admits ${\mathcal B}$ as a multiplicative basis.
%\end{example}

\medskip

The present paper is devoted to the study of Leibniz triple systems admitting a multiplicative basis, given special attention to its structure.
The paper is organized as follows. In $\S2$, we develop  connections techniques in the index set $K$ of the multiplicative basis so as to get a powerful tool for the study of this class of triple systems. By making use of these techniques we show that any Leibniz triple system $T$ admitting a multiplicative basis is  the orthogonal direct sum  $T = \bigoplus_{\alpha}{\mathcal I}_{\alpha}$ with each ${\mathcal I}_{\alpha}$ a well described ideal of $T$ admitting also a multiplicative basis. In $\S3$ the minimality of $T$ is characterized in terms of the multiplicative basis and it is shown that, in case the basis is $\mu$-multiplicative, the above decomposition of $T$ is actually by means of the family of its minimal ideals.

%%%%%%%%%%%%%%%%%%%%%%%%%%%%%%%%%%%%%%%%%%%%%%%%%%%%%%%%%%%%%%%%%%%%%%%%%%%%%%%%%%%%%%%%%%%%%%%%%%%%%%%%%%%%%%%%%%%%
\section{Connections in the index set. Decompositions}
%%%%%%%%%%%%%%%%%%%%%%%%%%%%%%%%%%%%%%%%%%%%%%%%%%%%%%%%%%%%%%%%%%%%%%%%%%%%%%%%%%%%%%%%%%%%%%%%%%%%%%%%%%%%%%%%%%%%

In what follows $T = {\frak I} \oplus {\neg \frak I}$ denotes a Leibniz triple system admitting a multiplicative basis over a base field ${\mathbb F}.$ We consider the multiplicative basis ${\mathcal B} = {\mathcal B}_{\frak I} \ud {\mathcal B}_{\neg \frak I}$ where ${\mathcal B}_{\frak I} = \{e_n\}_{n \in I}$ and ${\mathcal B}_{\neg \frak I}= \{u_r\}_{r \in J}.$ We begin this section by developing connection techniques among the elements in the index sets $I$ and $J$ as the main tool in our study.

By renaming if necessary we can suppose $I \cap J = \emptyset$. Now, for every $k \in {I \ud J}$, a new assistant variable $\bar{k} \notin {I \ud J}$ is introduced and we denote by $$\hbox{ $\bar{I}: = \{\bar{n}: n \in I\}$ and $\bar{J}: = \{\bar{r}: r \in J\}$}$$ the sets consisting of all these new symbols. Also, given every $\bar{k} \in \bar{I} \ud \bar{J}$ we will denote $$\overline{(\overline{k})} := k.$$
Finally, we will write by ${\mathcal P}(A)$ the power set of a given set $A$.

Next we introduce the following mappings which recover, in some sense, certain multiplicative relations among the elements of ${\mathcal B}:$

$$\hbox{$a : (I\;\dot\cup\;J) \times (I\;\dot\cup\;J) \times (I\;\dot\cup\;J) \to \mathcal{P}(I\;\dot\cup\;J)$}$$ defined by

\begin{itemize}
\item For $n \in I, r_1,r_2 \in J$,
$$a(n,r_1,r_2) := \{m\} \hspace{0.4cm} \hbox{if} \hspace{0.2cm}  0 \neq \{e_n, u_{r_1}, u_{r_2}\} \in \mathbb{F}e_m \hspace{0.2cm} \hbox{with} \hspace{0.2cm} m \in I$$

%$$\hbox{$a(n,r_1,r_2) := \left\{%
%\begin{array}{ll}
%\emptyset & \hbox{ if $\{e_n, u_{r_1}, u_{r_2}\} = 0$}\\
%\{m\} & \hbox{ if $0 \neq \{e_n, u_{r_1}, u_{r_2}\} \in \mathbb{F}e_m$ with $m \in I$}\\
%\end{array}%
%\right. $}$$

\item For $r_1,r_2,r_3 \in J,$
$$\hbox{$a(r_1,r_2,r_3):=\left\{%
\begin{array}{ll}
%\emptyset & \hbox{ if $\{u_{r_1}, u_{r_2}, u_{r_3}\} = 0$}\\
\{n\} & \hbox{if $0 \neq \{u_{r_1}, u_{r_2}, u_{r_3}\} \in \mathbb{F}e_n$ with $n \in I$}\\
\{s\} & \hbox{if $0 \neq \{u_{r_1}, u_{r_2}, u_{r_3}\} \in \mathbb{F}u_s$ with $s \in J$}\\
\end{array}%
\right. $}$$
\end{itemize}

\noindent and empty set in the remaining cases. We also consider $$\hbox{$b : (I\;\dot\cup\;J) \times (\overline{I}\;\dot\cup\;\overline{J}) \times (\overline{I}\;\dot\cup\;\overline{J}) \to \mathcal{P}(I\;\dot\cup\;J)$}$$ given as

\begin{itemize}
\item For $n \in I$ and $\overline{r}_1, \overline{r}_2 \in \overline{J},$
$$b(n,\overline{r}_1,\overline{r}_2):= \{m \in I : 0 \neq \{e_m, u_{r_1}, u_{r_2}\} \in \mathbb{F}e_n\} \cup \{s \in J : 0 \neq \{u_s, u_{r_1}, u_{r_2}\} \in \mathbb{F}e_n\}$$ $$\cup \{s \in J : 0 \neq \{u_{r_1},u_s, u_{r_2}\} \in \mathbb{F}e_n \} \cup \{s \in J : 0 \neq \{u_{r_1}, u_{r_2},u_s\} \in \mathbb{F}e_n\}$$

\item For $r \in J$ and $\overline{r}_1, \overline{r}_2 \in \overline{J},$
$$b(r,\overline{r}_1,\overline{r}_2):= \{s \in J : 0 \neq \{u_s, u_{r_1}, u_{r_2}\} \in \mathbb{F}u_r\} \cup \{s \in J : 0 \neq \{u_{r_1}, u_s, u_{r_2}\} \in \mathbb{F}u_r\}$$ $$\cup \{s \in J : 0 \neq \{u_{r_1}, u_{r_2}, u_s\} \in \mathbb{F}u_r\}$$
\end{itemize}

\noindent and empty set in the remaining cases.

Now we are in condition to introduce the map $$\mu : (I \ud J) \times (I \ud J \ud \overline{I} \ud \overline{J}) \times (I \ud J \ud \overline{I} \ud \overline{J}) \to \mathcal{P}(I \ud J)$$ given by

\medskip

\begin{itemize}
%\item $\mu(K_1,\overline{K}_2,K_3) = \mu(K_1,K_2,\overline{K}_3) = \emptyset$ being $K_1,K_2,K_3 \in \{I,J\}.$
\item $\mu(k_1,k_2,k_3):= \displaystyle\cup_{\sigma \in S_3} a(k_{\sigma(1)},k_{\sigma(2)},k_{\sigma(3)})$ for $k_1,k_2,k_3 \in I \ud J,$
\item $\mu(k,\overline{k}_1,\overline{k}_2) := \displaystyle\cup_{\sigma \in S_2} b(k,\overline{k}_{\sigma(1)},\overline{k}_{\sigma(2)})$ for $k \in I \ud J$ and $\overline{k}_1, \overline{k}_2 \in \overline{J},$
\end{itemize}
\noindent being $S_n$ the set of permutations of $n$ elements, and empty set in the remaining cases.

\begin{remark}\label{Remark}
From the definition of $\mu$ we have same trivial simmetries:
\begin{itemize}
\item $\mu(k_1,k_2,k_3)= \mu(k_{\sigma(1)},k_{\sigma(2)},k_{\sigma(3)})$ for $k \in I \ud J,$
\item $\mu(k, \overline{r}_1, \overline{r}_2) = \mu(k,\overline{r}_2, \overline{r}_1),$ with $k \in I \ud J$ and $\overline{r}_1, \overline{r}_2 \in \overline{J}$.
\end{itemize}
\end{remark}

\begin{lemma}\label{lema1}
Let $k_1,k_2 \in I \ud J$ and $h_1,h_2 \in J \ud \overline{J}$. Then, $k_1 \in \mu(k_2,h_1,h_2)$ if and only if $k_2 \in \mu(k_1,\overline{h}_1,\overline{h}_2)$.
\end{lemma}

\begin{proof}
By definition we have $\mu(I \ud J,J,\overline{J}) = \mu(I \ud J,\overline{J},J) =\emptyset$, then we just have two cases to consider. Let $k_1 \in \mu(k_2,h_1,h_2)$ and suppose $h_1,h_2 \in J$, meaning that, $k_1 \in \mu(k_2,h_1,h_2) = a(k_2,h_1,h_2) \cup a(k_2,h_2,h_1),$ and this is equivalent to $\{k_2\} \in b(k_1,\overline{h}_1,\overline{h}_2) \cup b(k_1,\overline{h}_2,\overline{h}_1).$ In any case $k_2 \in \mu(k_1,\overline{h}_1,\overline{h}_2)$ as wished. The case $h_1,h_2 \in \overline{J}$ is proved in a similar way.
\end{proof}

\medskip

Finally, we define the mapping $$\phi: \mathcal{P}(I \ud J) \times (I \ud J \ud \overline{I} \ud \overline{J}) \times (I \ud J \ud \overline{I} \ud \overline{J}) \to \mathcal{P}(I \ud J)$$ as $$\phi(K,p,q):=\bigcup_{k \in K} \mu(k,p,q)$$ for any $K \subset \mathcal{P}(I \ud J)$ and $p,q \in I \ud J \ud \overline{I} \ud \overline{J}$.

\begin{lemma}\label{lema2}
Consider $K \subset \mathcal{P}(I \ud J)$ and $p,q \in J \ud \overline{J}$. Then, $x \in \phi(K,p,q)$ if and only if $\phi(\{x\},\overline{p},\overline{q}) \cap K \neq \emptyset$.
\end{lemma}

\begin{proof}
We have $x \in \phi(K,p,q)$ if and only if there exists $k\in K$ such that $x \in \mu(k,p,q)$. By Lemma \ref{lema1}, this is equivalent to $k \in \mu(x,\overline{p},\overline{q}) = \phi(\{x\},\overline{p},\overline{q})$ and so also equivalent to $k \in \phi(\{x\},\overline{p},\overline{q}) \cap K \neq \emptyset.$
\end{proof}

\begin{definition}\label{connection}\rm
Let $k, k'$ be elements in $I \ud J$. We say that $k$ is {\it connected} to $k'$ if either $k = k'$ or there exists a subset $\{k_1,k_2,k_3,\dots,k_{2n+1}\} \subset I \ud J \ud \overline{I} \ud \overline{J}$ for some $n \geq 1$, such that the following conditions hold:
\begin{enumerate}
\item[{\rm 1.}] $k_1=k$.
\medskip
\item [{\rm 2.}] $\phi(\{k_1\},k_2,k_3) \neq \emptyset$,\\
$\phi(\phi(\{k_1\},k_2,k_3),k_4,k_5) \neq \emptyset$,\\
\hspace*{2cm} $\; \vdots $\\
$\phi(\phi(\cdots\phi(\{k_1\},k_2,k_3) \cdots), k_{2n-2}, k_{2n-1}) \neq \emptyset$.
\medskip
\item [{\rm 3.}]
$k' \in \phi(\phi(\cdots\phi(\{k_1\},k_2,k_3)\cdots), k_{2n}, k_{2n+1}).$
\end{enumerate}
The subset $\{k_1,k_2,k_3,\dots,k_{2n+1}\}$ is called a {\it connection} from $k$ to $k'$.
\end{definition}

\begin{remark}\label{Remark2}
Observe that Equation (\ref{equi}) and expressions of the applications $a$ and $b$ imply that the connection $\{k_2,k_3,\dots,k_{2n+1}\}$ must be included in $J \ud \overline{J}$.
\end{remark}

Our next goal is to show that the connection relation is an equivalence relation. Previously we show the next result.

\begin{lemma}\label{lema3}
Let $k,k' \in I \ud J$ with $k \neq k'$. If $\{k_1,k_2,k_3, \dots, k_{2n+1}\}$ is a connection from $k$ to $k'$, with $n \geq 1,$ then $\{k',\overline{k}_{2n+1},\overline{k}_{2n},\dots,\overline{k}_3,\overline{k}_2\}$ is a connection from $k'$ to $k$.
\end{lemma}

\begin{proof}
Let us prove it  by induction on $n$.

For $n=1$ we have that $k_1=k$ and $k' \in \phi(\{k\},k_2,k_3)$. Hence $k' \in \mu(k,k_2,k_3)$ and by Lemma \ref{lema1} and Remark \ref{Remark}, $k \in \mu(k',\overline{k}_2,\overline{k}_3) = \mu(k',\overline{k}_3,\overline{k}_2) = \phi(\{k'\},\overline{k}_3,\overline{k}_2)$. From here $\{k',\overline{k}_3,\overline{k}_2\}$ is a connection from $k'$ to $k$.

Let us suppose that the assertion holds for any connection with $2m+1$ elements and let us show it also holds for any connection $\{k_1,k_2,\dots,k_{2m+1},k_{2m+2},k_{2m+3}\}$ with $2m+3$ elements. If we denote $U:=\phi(\phi(\dots\phi(\{k_1\},k_2,k_3)\dots),k_{2m},k_{2m+1})$ then we have that $$k' \in\phi(U,k_{2m+2},k_{2m+3}).$$ From here, Lemma \ref{lema2} allows us to assert $\phi(\{k'\},\overline{k}_{2m+2},\overline{k}_{2m+3})\cap U \neq\emptyset$ and so we can take $u \in U$ such that by Remark \ref{Remark}
\begin{equation}\label{eqq1}
u \in \phi(\{k'\},\overline{k}_{2m+2},\overline{k}_{2m+3}) = \phi(\{k'\},\overline{k}_{2m+3},\overline{k}_{2m+2}).
\end{equation}
Since $u \in U$ we have that $\{k_1, k_2, \dots, k_{2m}, k_{2m+1}\}$ is a connection from $k$ to $u$ and so, by induction hypothesis, the set $\{u,\overline{k}_{2m+1},\overline{k}_{2m},\dots,\overline{k}_3,\overline{k}_2\}$ is a connection from $u$ to $k$. This fact together with   Equation (\ref{eqq1}) give us that $$k \in\phi(\phi(\dots\phi(\phi(\{k'\},\overline{k}_{2m+3},\overline{k}_{2m+2}),\overline{k}_{2m+1},\overline{k}_{2m})\dots),\overline{k}_3,\overline{k}_2).$$
Hence $\{k',\overline{k}_{2m+3},\overline{k}_{2m+2},\dots,\overline{k}_3,\overline{k}_2\}$ is a connection from $k'$ to $k$ and the proof is complete.
\end{proof}

\begin{proposition}
The relation $\sim$ in $I \ud J$, defined by $k \sim k'$ if and only if $k$ is connected to $k'$, is an equivalence relation.
\end{proposition}

\begin{proof}
The reflexive and symmetric character of $\sim$ are given by Definition \ref{connection} an Lemma \ref{lema3} respectively.
To check the transitivity character, observe that if we consider a couple of  connections $\{k_1,\dots,k_{2n+1}\}$ and $\{{k'}_1,\dots,{k'}_{2m+1}\}$ from $k$ to $k'$ and from $k'$ to $k''$ respectively, then the set $\{k_1,\dots,k_{2n+1},{k'}_2,\dots,{k'}_{2m+1}\}$ is a connection from $k$ to $k''$.
\end{proof}

By the above proposition we can introduce the quotient set
$$(I \ud J)/ \sim := \{[k] : k \in I \ud J\},$$ becoming $[k]$ the set of elements in $I \ud J$ which are connected to $k$. Our next purpose is to construct an adequate ideal in $T$ associated to each $[k] \in (I \ud J)/\sim$. We define the linear subspace $$T_{[k]} := \Bigl(\bigoplus\limits_{n \in [k]\cap I} \mathbb{F}e_n \Bigr) \oplus \Bigl(\bigoplus\limits_{r \in [k]\cap J} \mathbb{F}u_r \Bigr).$$

%with at least one $0 \neq e_n \in \mathcal{B}_{\frak I},$ because any ideal in $T$ must have non-zero intersection with $\frak I$.
\smallskip

\begin{lemma}\label{lema4}
If $\{T_{[k]}, T, T_{[h]}\} + \{T,T_{[k]}, T_{[h]}\} \neq 0$ for some $[k],[h] \in (I \ud J)/\sim$, then $[k]=[h]$ and moreover $\{T_{[k]}, T, T_{[h]}\} + \{T,T_{[k]}, T_{[h]}\} \subset T_{[k]}.$
\end{lemma}

\begin{proof}
Suppose $\{T_{[k]}, T, T_{[h]}\} \neq 0,$ then we have three possibilities:
\begin{itemize}
\item There exist $n_0 \in [k] \cap I, p \in J$ and $s_0 \in [h] \cap J$ such that $0 \neq \{e_{n_0},u_p,u_{s_0}\} \in \mathbb{F}e_n$ for some $n \in I.$ Then $n \in \mu(n_0,p,s_0) = \phi(\{n_0\},p,s_0)$ and $\{n_0,p,s_0\}$ is a connection from $n_0$ to $n$. From here $k \sim n_0 \sim n$. Observe that by Remark \ref{Remark} $\mu(n_0,p,s_0) = \mu(s_0,p,n_0),$ therefore $s_0 \sim n$. From $h \sim s_0 \sim n$ we conclude $[h] = [k]$.

\item There exist $r_0 \in [k] \cap J, s_0 \in [h] \cap J$ and $p \in J$ such that $0 \neq \{u_{r_0},u_{s_0},u_p\} \in \mathbb{F}e_n$ for some $n \in I,$ or $0 \neq \{u_{r_0},u_{s_0},u_p\} \in \mathbb{F}u_r$ for some $r \in J.$

\begin{itemize}
\item In first case by Remark \ref{Remark} we have that $n \in \mu(r_0,s_0,p) = \mu(s_0,r_0,p)$. Hence $n \in \phi(\{r_0\},s_0,p)$ and $n \in \phi(\{s_0\},r_0,p)$. Considering the connections $\{r_0,s_0,p\}$ and $\{s_0,r_0,p\}$ we have respectively that $n \sim r_0$ and $n \sim s_0,$ then $k \sim r_0 \sim n \sim s_0 \sim h.$

\item Second case is analogous, being $r \in \mu(r_0,s_0,p) = \mu(s_0,r_0,p)$ by Remark \ref{Remark}, and then $r \in \phi(\{r_0\},s_0,p)$ and $r \in \phi(\{s_0\},r_0,p)$. Using the connections $\{r_0,s_0,p\}$ and $\{s_0,r_0,p\}$ we get $r \sim r_0$ and $r \sim s_0$ respectively, and again $k \sim r_0 \sim r \sim s_0 \sim h.$
\end{itemize}
\end{itemize}
Similarly we can proceed with $\{T,T_{[k]},T_{[h]}\} \neq 0$ and we obtain the result.
\end{proof}

\begin{definition}\rm
Let $T = {\frak I} \oplus {\neg \frak I}$ be a Leibniz triple system with multiplicative basis ${\mathcal B}.$ It is said that a subtriple $S$ of $T$ admits a multiplicative basis $\mathcal{B}_S$ {\it inherited} from $\mathcal{B}$ if $\mathcal{B}_S$ is a multiplicative basis of $S$ satisfying $\mathcal{B}_S \subset \mathcal{B}$.
\end{definition}

\begin{proposition}\label{abril}
For any $[k] \in (I \ud J)/\sim$, the linear subspace $T_{[k]}$ is an ideal of $T$ admitting a multiplicative basis inherited from the one of $T$.
\end{proposition}
\begin{proof}
Since $T=\Bigl(\bigoplus\limits_{n \in I} \mathbb{F}e_n \Bigr) \oplus \Bigl(\bigoplus\limits_{r \in J} \mathbb{F}u_r \Bigr)$ we clearly have $T = \bigoplus\limits_{[h] \in (I \ud J)/\sim} T_{[h]}$ and so we can write $$\{T_{[k]},T, T\} + \{T,T_{[k]}, T\} + \{T,T,T_{[k]}\} = $$
$$\{T_{[k]},T,\bigoplus_{[h] \in (I \ud J)/\sim} T_{[h]}\} + \{T, T_{[k]},\bigoplus_{[h] \in (I \ud J)/\sim} T_{[h]}\} + \{\bigoplus_{[h] \in (I \ud J)/\sim} T_{[h]},T, T_{[k]}\}$$ $$\subset T_{[k]},$$ being last inclusion consequence of Lemma \ref{lema4}. That is, any $T_{[k]}$ is actually an ideal of $T$. Finally, observe that the set $$\{e_m : m \in [k]\cap I\} \ud \{u_s : s \in [k]\cap J\}$$ is a multiplicative basis of $T_{[k]}$ inherited from the basis ${\mathcal B} = {\mathcal B}_{\frak I} \ud {\mathcal B}_{\neg \frak I},$ with ${\mathcal B}_{\frak I} = \{e_n\}_{n \in I}$ and ${\mathcal B}_{\neg \frak I}= \{u_r\}_{r \in J},$ of $T$.
\end{proof}

\begin{theorem}\label{theo1}
Let $T$ be a Leibniz triple system admitting a multiplicative basis ${\mathcal B}$. Then $T$ is the orthogonal direct sum of the ideals
$$T = \bigoplus_{[k] \in (I \ud J)/\sim} T_{[k]},$$ admitting each $T_{[k]}$ a multiplicative basis inherited from $\mathfrak{B}$.
\end{theorem}

\begin{proof}
As in Proposition \ref{abril}, the fact $\displaystyle T = \Bigl(\bigoplus_{n \in I} \mathbb{F}e_n \Bigr) \oplus \Bigl(\bigoplus_{r \in J} \mathbb{F}u_r \Bigr)$ gives us $$T = \bigoplus_{[k] \in (I \ud J)/\sim} T_{[k]}.$$
Applying again Proposition \ref{abril} we have that each $T_{[k]}$ is an ideal admitting a multiplicative basis inherited, and by Lemma \ref{lema4} we get the orthogonality of the previous direct sum.
\end{proof}

%\begin{corollary}\label{coro1}
%If $T$ is simple, then there exists a connection between any two elements of $I$.
%\end{corollary}

%\begin{proof}
%The simplicity of $T$ applies to get that $T_{[i]} = T$ for some $[i]\in I/\sim$. Hence $[i] = I$ and so any couple of elements in $I$ are connected.
%\end{proof}

%%%%%%%%%%%%%%%%%%%%%%%%%%%%%%%%%%%%%%%%%%%%%%%%%%%%%%%%%%
\section{The minimal components}
%%%%%%%%%%%%%%%%%%%%%%%%%%%%%%%%%%%%%%%%%%%%%%%%%%%%%%%%%%

In this section our goal is to characterize the minimality of the ideals which give rise to the decomposition of $T$ in Theorem \ref{theo1}, in terms of connectivity properties in the index set $I \ud J$.

\begin{definition}\rm
A Leibniz triple system $T$ admitting a multiplicative basis $\mathcal{B}$ is called {\it minimal} if its only nonzero ideals admitting a multiplicative basis inherited from $\mathcal{B}$ are $\frak I$ or $T$.
\end{definition}

Let us introduce the notion of $\mu$-multiplicativity in the framework of Leibniz triple systems admitting a multiplicative basis, in a similar way
%Leibniz superalgebras admitting a multiplicative basis \cite{Leibniz_mult} and
to the ones of closed-multiplicativity for Lie triple systems (see \cite{Yo_simple_triple, Yo_split_triple, Yo_Forero2}), for split 3-Lie algebras (see \cite{3-Lie}), or for different classes of algebras like split Leibniz algebras or split Lie color algebras (see \cite{Yo3,Yo6} for these notions and examples).

\begin{definition}\label{defmulti}\rm
It is said that a Leibniz triple system $T$ admits a $\mu${\it-multiplicative basis} $\mathcal{B} = \mathcal{B}_{\frak I} \ud \mathcal{B}_{\neg \frak I},$ where $\mathcal{B}_{\frak I} = \{e_i\}_{i \in I}$ and $\mathcal{B}_{\neg \frak I} = \{u_r\}_{r \in J},$ if $\mathcal{B}$ is multiplicative and given $t_1, t_2 \in I \ud J$ and $s_1,s_2 \in J$ (resp. $s_1,s_2 \in \overline{J}$) such that $t_2 \in \mu(t_1,s_1,s_2),$ then  $v_{t_2} \in \mathbb{F}\{v_{t_1} ,u_{s_{\sigma(1)}},u_{s_{\sigma(2)}}\}$ (resp. $v_{t_2} \in \mathbb{F}\{v_{t_1},u_{\overline{s}_{\sigma(1)}},u_{\overline{s}_{\sigma(2)}}\}$) for some $\sigma \in S_2$, where any $v_{t_i} = e_{t_i}$ or $v_{t_i} = u_{t_i}$ depending on $t_i \in I$ or $t_i \in J$, $ i \in \{1,2\}$.
\end{definition}

\begin{proposition}\label{nueva12}
Suppose $T$ admits a $\mu$-multiplicative basis ${\mathcal B}.$ If $J$ has all of its elements connected, then any nonzero ideal ${\mathcal I}$ of $T$ with a multiplicative basis inherited from ${\mathcal B}$ such that ${\mathcal I} \not\subset {\frak I}$ satisfies ${\mathcal I} = T$.
\end{proposition}

\begin{proof}
Since ${\mathcal I} \not\subset {\frak I}$ we can take some $r_1 \in J$ such that
\begin{equation}\label{as}
0 \neq u_{r_1} \in  {\mathcal I}.
\end{equation}
We know that $J$ has all of their elements connected. If $J = \{r_1\}$ trivially $\neg \frak I \subset \mathcal{I}$. If $|J| > 1$ we take $s \in J \setminus \{r_1\}$, being then $0 \neq u_s$, we can consider a connection
\begin{equation}\label{5}
\{r_1,r_2,...,r_{2n+1}\} \subset J \ud \overline{J}
\end{equation}
from $r_1$ to $s$.

We know that the subset $\phi(\{r_1\}, r_2,r_3)$ of $I \ud J$ is non empty, and we can take $a_1 \in \phi(\{r_1\}, r_2,r_3) = \mu(r_1,r_2,r_3)$. Now taking into account Equation (\ref{as}) and the $\mu$-multiplicativity of ${\mathcal B}$ we get either
$$0 \neq v_{a_1} \in  {\mathbb F}\{u_{r_{\sigma(1)}}, u_{r_{\sigma(2)}}, u_{r_{\sigma(3)}}\} \subset {\mathcal I}$$ for some $\sigma \in S_3$, or
$$0 \neq v_{a_1} \in {\mathbb F}\{u_{r_1}, u_{\overline{r}_{2}}, u_{\overline{r}_{3}}\} + {\mathbb F}\{u_{r_1},u_{\overline{r}_{3}}, u_{\overline{r}_{2}}\}\subset {\mathcal I}$$ where $v_{a_1} = e_{a_1}$ or $v_{a_1} = u_{a_1}$ depending on $a_1 \in I$ or $a_1 \in J$.

Since $s \in J$, necessarily $\phi(\{r_1\},r_2,r_3) \cap J \neq \emptyset$ and we have
\begin{equation}\label{7}
0 \neq \bigoplus\limits_{r \in \phi(\{r_1\}, r_2,r_3) \cap J} {\mathbb F}u_r \subset {\mathcal I}.
\end{equation}
Since $$\phi(\phi(\{r_1\}, r_2, r_3),r_4,r_5) \neq \emptyset$$ we can argue as above, considering Equation (\ref{7}), to get $$0 \neq \bigoplus\limits_{r \in \phi(\phi(\{r_1\},r_2,r_3),r_4,r_5) \cap J} {\mathbb F}u_r \subset {\mathcal I}.$$ By reiterating this process with the connection (\ref{5}) we obtain $$0 \neq \bigoplus\limits_{r \in \phi(\phi (\cdots (\phi(\{r_1\}, r_2,r_3), \cdots), r_{2n-2}), r_{2n-1}) \cap J} {\mathbb F}u_r \subset {\mathcal I}.$$ Since $s \in \phi(\phi (\cdots (\phi(\{r_1\},r_2,r_3),\cdots),r_{2n}),r_{2n+1}) \cap J$ we conclude
$u_s \in {\mathcal I}$ for all $s \in J \setminus \{r_0\}$ and so
\begin{equation}\label{b1}
{\neg \frak I} = \bigoplus\limits_{r \in {J}}{\mathbb F} u_r \subset {\mathcal I}.
\end{equation}
By Equations (\ref{gentle}) and (\ref{equi}) we have that ${\frak I} \subset \{{\frak I}, {\neg \frak I}, {\neg \frak I}\}+ \{{\neg \frak I}, {\neg \frak I}, {\neg \frak I}\}$. Then Equation (\ref{b1}) allows us to assert
\begin{equation}\label{b2}
{\frak I} \subset {\mathcal I}.
\end{equation}
Finally, since $T = {\frak I} \oplus \neg \frak I$, Equations (\ref{b1}) and (\ref{b2}) give us $ {\mathcal I} = T$.
\end{proof}

\begin{proposition}\label{propoI}
Suppose $T$ admits a $\mu$-multiplicative basis ${\mathcal B}$ and $I$ has all of its elements connected, then any nonzero ideal ${\mathcal I}$ of $T$ with a multiplicative basis inherited from ${\mathcal B}$ such that ${\mathcal I} \subset {\frak I}$ satisfies ${\mathcal I} = {\frak I}$.
\end{proposition}

\begin{proof}
Taking into account ${\mathcal I} \subset {\frak I}$ we can fix some $n_0 \in I$ satisfying $$0 \neq e_{n_0} \in {\mathcal I}.$$ Since $I$ has all of its elements connected, we can argue from $n_0$ with the $\mu$-multiplicativity of ${\mathcal B}$ as it is done in Proposition \ref{nueva12} from $r_0$ to get ${\frak I} \subset {\mathcal I}$ and then ${\mathcal I}={\frak I}$.
\end{proof}

\begin{theorem}\label{theo2}
Let $T$ be a Leibniz triple system admitting a $\mu$-multiplicative basis $\mathcal{B}.$ The Leibniz triple system $T$ is minimal if and only if the index sets $I,J$ have all of its elements connected respectively.
\end{theorem}

\begin{proof}
Suppose $T$ is minimal. From Proposition \ref{abril} we have that all nonzero ideal $T_{[k]}$ must be $\frak I$ or $T$. We distinguish two possibilities. If $\neg \frak I = \emptyset$, all $T_{[k]} = \frak I = T$ and all elements in $I$ are connected. If $\neg \frak I \neq \emptyset$, suppose all $T_{[k]} = \frak I,$ then we get a contradiction with $T = \frak I \oplus \neg \frak I$ and Theorem \ref{theo1}. Therefore exists at least one $[k_0] \in (I \ud J)/\sim$ such that $T_{[k]} = T,$ then $[k] = I \ud J$ and so any couple of elements in $I \ud J$ are connected, in particular $I,J$ have all of their  elements connected respectively.

Conversely, consider a nonzero ideal $\mathcal{I}$ of $T$ admitting a multiplicative basis inherited from $\frak B$. If $\mathcal{I} \subset \frak I,$ from $I$ has all its elements connected, by Proposition \ref{propoI} is $\mathcal{I} = \frak I$. Otherwise $\mathcal{I} \not\subset \frak I$, from $J$ has all its elements connected, by Proposition \ref{nueva12} we have $\mathcal{I} = T$.
\end{proof}

\begin{theorem}
Let $T$ be a Leibniz triple system  admitting a $\mu$-multiplicative basis $\mathfrak B$. Then $T$ is the orthogonal direct sum of the family of its minimal ideals $$T=\bigoplus_k \mathfrak I_k,$$ each one admitting a $\mu$-multiplicative basis inherited from $\mathfrak B$.
\end{theorem}

\begin{proof}
By Theorem \ref{theo1} we have that $T=\displaystyle\bigoplus_{[k] \in (I \ud J)/\sim} T_{[k]}$ is the orthogonal direct sum of the ideals $T_{[k]}$, admitting each $T_{[k]}$ a multiplicative basis $\mathfrak B_{[k]} := \{e_n: n \in [k]\cap I\} \ud \{u_r: r \in [k]\cap J\}$ inherited from $\mathfrak B$.

We wish to apply Theorem \ref{theo2} to each $T_{[k]}$ so we are going to verify that the basis $\mathfrak B_{[k]}$ is $\mu$-multiplicative and that $[k]$ has all of its elements  $[k]$-connected, that is, connected through elements contained in $[k]$ or its respective symbols $\overline{[k]} := \{\overline{h} : h \in [k]\}$.

\noindent We clearly have that $T_{[k]}$ admits $\mathfrak B_{[k]}$ as $\mu$-multiplicative basis as consequence of having a basis inherited from $\mathfrak B$ and the fact $\{T_{[k]},T, T_{[h]}\} + \{T,T_{[k]},T_{[h]}\}  = 0$ when $[k] \neq [h]$.

Finally, let $k',k'' \in [k]$ and consider a connection from $k'$ to $k''$
\begin{equation}\label{eqq7}
\{k',j_2,\dots,j_{2n+1}\}.
\end{equation}

\noindent Let $x \in \phi(\{k'\},j_2,j_3)$. Then $x \in [k].$ We are going to check that the connection $\{k',j_2,j_3\}$ satisfies $k',j_2,j_3 \in [k]$. Clearly $k' \in [k]$ and we have two possibilities:

\begin{itemize}
\item If $j_2,j_3 \in J$, we also have, see Remark \ref{Remark}, that $x \in \phi(\{j_2\},k',j_3)$ and $x \in \phi(\{j_3\},k',j_2)$. From here, the connections $\{j_2,k',j_3\}$ and $\{j_3,k',j_2\}$ give us $j_2 \sim x$ and $j_3 \sim x,$ respectively. Hence $j_2,j_3 \in [k]$.

\item If $j_2,j_3 \in \overline{J}$, by Lemma \ref{lema1} we have $k' \in \phi(\{x\},\overline{j}_2,\overline{j}_3)$. Arguing as above we get $\overline{j}_2,\overline{j}_3 \in {[k]}$ and so $j_2, j_3 \in \overline{[k]}$.
\end{itemize}
By iterating this process we obtain that all of the elements in the connection (\ref{eqq7}) are contained in $[k] \ud \overline{[k]}$. That is, $[k]$ has all of its elements $[k]$-connected. From the above, we can apply Theorem \ref{theo2} to any $T_{[k]}$ so as to conclude $T_{[k]}$ is minimal.

\noindent It is clear that the decomposition $T = \displaystyle\bigoplus_{[k] \in (I \ud J)/\sim} T_{[k]}$ satisfies the assertions of the theorem and the proof is completed.
\end{proof}

\medskip

\begin{remark}\rm
We recall that an {\it arbitrary triple system} ${\frak T}$  is just a vector space endowed with a trilinear map  $$\langle \cdot, \cdot, \cdot \rangle: {\frak T} \times {\frak T} \times {\frak T} \to {\frak T},$$ called its {\it triple product}, and where any identity, (associativity, Lie, Jordan, Leibniz, etc.), is not supposed  to be satisfied by the triple product.

\medskip

We would like to know that the identities which define a Leibniz triple system, given in Definition \ref{deffi}, are only used in the development of the preset work to verify that  the ideal $ \frak I$, (see Equation (\ref{gentle})),  satisfy  the identities (\ref{equi}). From here, the results getting in this paper not only hold for the class of Leibniz triple systems, but also hold for any arbitrary  triple system ${\frak T}$  such that can be written as the direct sum of two linear subspaces $${\frak T}= {\frak I} \oplus V$$ where ${\frak I}$ is an ideal of ${\frak T}$ satisfying $$ \langle {\frak T}, {\frak I}, {\frak T} \rangle  = \langle {\frak T}, {\frak T}, {\frak I}\rangle = 0.$$
\end{remark}

%{\bf Acknowledgment.} We would like to thank  the referee  for the detailed reading of this work and for the suggestions which have improved the final version of the same.


\begin{thebibliography}{99}

\bibitem{Leibniz2} Albeverio, S.; Ayupov, S.; Omirov, B.A.: On nilpotent and simple Leibniz algebras. {\it Comm. Algebra} 33, no. 1, (2005), 159--172.

\bibitem{Leibniz4} Ayupov, S.; Omirov, B.A.: On Leibniz algebras. Algebra and operator theory. Kluwer Acad. Publ., Dordrecht, (1998), 1--12.

\bibitem{re1} Balogh, Z.: Further results on a filtered multiplicative basis of group algebras. {\it Math. Commun.} 12, no. 2, (2007), 229-238.

\bibitem{m5} Bautista, R., Gabriel, P., Roiter, A.V. and Salmeron, L.: Representation-finite algebras and multiplicative basis. {\it Invent. Math.} 81, (1985), 217--285.

\bibitem{Ref} Bovdi, V., Grishkov, A. and  Siciliano, S.: On filtered multiplicative bases of some associative  algebras. {\it Algebr. Represent. Theory} 18, no. 2, (2015), 297--306.

\bibitem{Bremner} Bremner, M.; S\'anchez-Ortega, J.: Leibniz triple systems. {\it Commun. Contemp. Math.} 16, no. 1, (2014), 1350051, 19 pp.

\bibitem{Yo_simple_triple} Calder\'on A.J.: On simple split Lie triple systems. {\it Algebr. Represent. Theory} 12, (2009), 401--415.

\bibitem{Yo_integrable} Calder\'on A.J.: Integrable roots in split Lie triple systems. {\it Acta Math. Sinica} 25, (2009), 1759--74.

\bibitem{Yo_split_triple} Calder\'on A.J.: On split Lie triple systems. {\it Proc. Indian Acad. Sci. Math. Sci.} 119, no. 2, (2009), 165--177.

\bibitem{Arb_Leibniz_Super} Calder\'on, A.J.: Leibniz algebras admitting a multiplicative basis. {\it Bull. Malaysian Math. Soc.} In press.

\bibitem{Yo_Forero2} Calder\'on, A.J. and Forero, M.: On split Lie triple systems II. {\it Proc. Indian Acad. Sci. Math. Sci.} 120, no. 2, (2010), 185--198.

\bibitem{3-Lie} Calder\'on, A.J. and Forero, M.: Split 3-Lie algebras. {\it J. Math. Phys.} 52, (2011), 123503, 16 pp.

\bibitem{Yo_Arb_Alg} Calder\'on, A.J.; Navarro, F.J.: Arbitrary algebras with a multiplicative basis. {\it Linear Algebra Appl.} 498, (2016), 106--116.

\bibitem{Yo3} Calder\'on, A.J. and S\'{a}nchez, J.M.: Split  Leibniz algebras. {\it Linear Algebra Appl.} 436, no. 6,  (2012), 1648--1660.

\bibitem{Yo6} Calder\'on, A.J. and S\'{a}nchez, J.M.:  On the structure of split  Lie color algebras. {\it Linear Algebra Appl.} 436, no. 2, (2012), 307--315.

\bibitem{huel2}  Camacho, L.M., Gomez, J.R., Gonz\'{a}lez, A.J. and Omirov, B.A.: Naturally graded 2-filiform Leibniz
algebras. {\it Comm. Algebra.} 38, (2010), 3671--3685.

\bibitem{le18}  Ca${\rm\tilde{n}}$ete, E.M. and Khudoyberdiyev, A.Kh.: The classification of 4-dimensional Leibniz algebras. {\it Linear Algebra Appl.} 439, no. 1, (2013), 273--288.

\bibitem{le16} Casas, J.M., Khudoyberdiyev, A.Kh., Ladra, M. and Omirov, B.A.: On the degenerations of solvable Leibniz algebras. {\it Linear Algebra Appl.} 439, no. 2, (2013), 472--487.

\bibitem{le24}  Casas, J.M., Ladra, M., Omirov, B.A. and
Karimjanov, I.A.: Classification of solvable Leibniz algebras with null-filiform nilradical. {\it Linear Multilinear Algebra} 61, no. 6, (2013), 758--774.

\bibitem{le22}  Casas, J.M. Ladra, M., Omirov, B.A. and Karimjanov, I.A. Classification of solvable Leibniz algebras with naturally graded filiform nilradical. {\it Linear Algebra Appl.} 438, no. 7, (2013), 2973--3000.

\bibitem{Jacobson} Jacobson, N.: Structure and representations of Jordan algebras. American Math. Society, Providence, (1968).

\bibitem{Loday} Loday, J.L.: Une version non commutative des
 alg\`{e}bres de Lie: les alg\`{e}bres de Leibniz. {\it Enseign. Math. (2)} 39, no. 3-4, (1993), 269--293.

\bibitem{huel3} Omirov, B.A., Rakhimov, I.S. and Turdibaev, R.M.: On description of Leibniz algebras corresponding to ${sl}_2$. {\it Algebr. Represent. Theory} 16 (2013), 1507--1519.
\end{thebibliography}
\end{document}